\documentclass[10pt]{amsart}

\usepackage{amsmath, amsthm}
\usepackage{eucal}
\usepackage{fullpage}

\usepackage{amssymb}
\usepackage{amscd}
\usepackage{palatino}
\usepackage{latexsym}
\usepackage{epsfig}
\usepackage{graphicx}
\usepackage{amsfonts}
\usepackage{psfrag}
\usepackage{color}
\usepackage{curves}

\usepackage{url}
\usepackage{cite}

\usepackage[margin=1in]{geometry}
\usepackage[pdftex]{hyperref}

\numberwithin{equation}{section}
\numberwithin{figure}{section}

\newtheorem{theorem}{Theorem}[section]
\newtheorem{lemma}[theorem]{Lemma}
\newtheorem{proposition}[theorem]{Proposition}

\newtheorem{remark}[theorem]{Remark}

\newtheorem{example}[theorem]{Example}

\theoremstyle{definition}
\newtheorem{definition}[theorem]{Definition}

\newcommand{\al}{\alpha}
\newcommand{\be}{\beta}
\newcommand{\la}{\lambda}
\newcommand{\si}{\sigma}

\newcommand{\de}{\delta}

\newcommand{\lee}{\langle}
\newcommand{\ree}{\rangle}

\newcommand{\mf}{\mathbf}
\newcommand{\mr}{\mathrm}

\newcommand{\C}{{\mathbb{C}}}

\newcommand{\Z}{{\mathbb{Z}}}

\newcommand{\R}{{\mathbb{R}}}

\newcommand{\mfc}{\mf{c}}
\newcommand{\mfl}{\mf{\ell}}

\renewcommand{\t}{\mathfrak{t}}

\begin{document}

\title{Grossberg-Karshon twisted cubes and basepoint-free divisors}

\author{Megumi Harada}
\address{Department of Mathematics and
Statistics\\ McMaster University\\ 1280 Main Street West\\ Hamilton, Ontario L8S4K1\\ Canada}
\email{Megumi.Harada@math.mcmaster.ca}
\urladdr{\url{http://www.math.mcmaster.ca/Megumi.Harada/}}

\author{Jihyeon Jessie Yang}
\address{Department of Mathematics and
Statistics\\ McMaster University\\ 1280 Main Street West\\ Hamilton, Ontario L8S4K1\\ Canada}
\email{jyang@math.mcmaster.ca}
\urladdr{\url{http://www.math.mcmaster.ca/~jyang}}
\thanks{}

\date{\today}

\begin{abstract}
Let $G$ be a complex semisimple simply connected linear algebraic
group.
The main result of this note is to give several equivalent criteria for the
\emph{untwistedness} of the \emph{twisted cubes}
introduced by Grossberg and Karshon. In certain cases arising from
representation theory, Grossberg and Karshon obtained a Demazure-type character
formula for irreducible $G$-representations
as a sum over lattice
points (counted with sign according to a density function)
of these twisted cubes. A twisted cube is untwisted when it is a
``true'' (i.e. closed, convex) polytope; in this case,
Grossberg and Karshon's character formula becomes a purely \emph{positive}
formula with no multiplicities, i.e. each
lattice point appears precisely once in the formula, with coefficient $+1$.
One of our
equivalent conditions for untwistedness is that a certain
divisor on the special fiber of a
toric degeneration of a
Bott-Samelson variety, as constructed by Pasquier, is
basepoint-free.
We also show that the
strict positivity of some of the defining constants for the
twisted cube, together with convexity (of its
support), is enough to guarantee untwistedness.
 Finally, in the special case when the twisted
cube arises from the representation-theoretic data of $\lambda$ an
integral weight and $\underline{w}$ a choice of word
decomposition of a Weyl group element, we give two simple necessary
conditions for untwistedness which is stated in terms of $\lambda$ and
$\underline{w}$.
\end{abstract}


\maketitle

\setcounter{tocdepth}{1}

\tableofcontents
\section*{Introduction}

Constructing a combinatorial model for a basis of a representation
is a fruitful technique in modern representation theory, as exhibited
by the famous theory of crystal bases and string polytopes. Another
well-known example from toric geometry is the bijective correspondence
between the lattice points in the moment polytope of a nonsingular projective toric variety $X$
with a basis consisting of $T$-weight vectors of the space $H^0(X,L)$ of holomorphic
sections of the prequantum line bundle $L$ over $X$. These two examples
are linked via toric degenerations: Kaveh \cite{Kaveh} recently showed that string
polytopes can be obtained as Okounkov bodies of flag varieties $G/B$,
and using results of Anderson\cite{Anderson} one can show that in this case there is a toric degeneration of
$G/B$ to a toric variety $X$ whose corresponding polytope is the string polytope.

In much earlier work (from the 1990s) Grossberg and Karshon \cite{Grossberg-Karshon} also
constructed  degenerations of complex structures from Bott-Samelson varieties to toric varieties
(specifically,  Bott towers) and consequently obtained a Demazure-type character
formula for irreducible representations of complex semisimple
(simply-connected) linear algebraic groups $G$. Their character
formula
can be combinatorially interpreted in terms of \emph{twisted cubes}
(cf. Definition~\ref{definition:twisted cube}).
These twisted cubes are combinatorially simpler than general string polytopes but they are not ``true''
polytopes in the sense that their faces may have various angles and the
intersection of faces may not be a face (cf. \cite[\S 2.5 and Figure
1 therein]{Grossberg-Karshon}),
and (the support of the) twisted cube may be neither
closed nor convex. An example is shown in Figure~\ref{figure:first
  example}.
In particular, 
the Grossberg-Karshon character formula takes a sum
over a set of integer lattice points, where a summand may appear with
either a plus or
minus sign
\cite[Theorem 6 and
Remark 3.22]{Grossberg-Karshon}. In this sense, the Grossberg-Karshon character formula is not,
in general,
a purely combinatorial `positive' formula.

The main result of this note, recorded in Theorem~\ref{theorem:Untwistedness and Base-point free divisor}, gives
 several equivalent conditions for the
Grossberg-Karshon twisted cubes to be  \emph{untwisted}
(cf. Definition~\ref{definition:untwisted}), i.e., it is a ``true''
(closed, convex) polytope.
In the representation-theoretic
applications of Grossberg and Karshon, this corresponds to the case
when
the Grossberg-Karshon character formula is a `positive' formula.
One of our equivalent conditions can be stated naturally in
terms of the toric variety $X(\mfc)$ obtained as the special fiber of a toric
degeneration of Bott-Samelson varieties as constructed by Pasquier \cite{Pasquier}
(which was in turn motivated from the degeneration of complex
structures from a Bott-Samelson variety to a Bott tower given in
\cite{Grossberg-Karshon}). More specifically, the condition is that a
certain torus-invariant divisor $D(\mfc, \mfl)$
on the toric variety $X(\mfc)$ is \emph{basepoint-free}. By some standard
results in toric geometry, the basepoint-free-ness of a divisor can also be
stated in terms of the \emph{Cartier data} of the divisor, and this in
turn provides us with computationally efficient methods for
determining the untwistedness of the Grossberg-Karshon twisted cube.
We note that the relationship between untwistedness and basepoint-free-ness
seems to have already been known by experts (see e.g. the comments in \cite[\S
6]{Alexeev-Brion}).
However, our computationally effective
characterization of untwistedness in terms of the Cartier data of
$D(\mfc, \mfl)$ depends on an explicit description of the line bundle
$\mathcal{L}_0$
(over the special fiber of the toric degeneration) as
$\mathcal{O}(D(\mfc, \mfl))$, i.e., the line bundle corresponding to
the divisor $D(\mfc, \mfl)$, as well as a concrete computation of
$D(\mfc, \mfl)$ \cite{Harada-Yang}.
As far as we are aware, these results
do not appear in the previous literature.

In Section~\ref{sec:positivity} we also record
the observation that the convexity of the twisted cube, and the strict
positivity of a certain subset of the
constants used in the definition of the twisted cube, are sufficient to guarantee
its untwistedness.
Finally, in Section~\ref{sec:G/B} we give two simple
necessary conditions for the untwistedness of the Grossberg-Karshon
twisted cube $(C, \rho)$ when it arises from the
representation-theoretic data of a weight $\lambda \in \t^*_\Z$
and a choice of reduced word decomposition $\underline{w}$ for an
element $w$ in the Weyl group of $G$.
The two necessary conditions are not sufficient,
 as can be seen by a simple example (cf. Example~\ref{example:eunjeong}); a full discussion
 of necessary and sufficient conditions for untwistedness, stated in
 terms of the representation theoretic data, is in
 \cite{HaradaLee}.

Finally, we remark that Kiritchenko recently has defined \emph{divided
  difference operators $D_i$ on polytopes} and, using these $D_i$
inductively and a fixed choice of reduced word decomposition
for the longest element in the Weyl group of $G$, she constructs (possibly virtual) polytopes whose lattice
points encode the character of irreducible $G$-representations.
\cite[Theorem 3.6]{Kiritchenko}.
Kiritchenko's virtual polytopes simultaneously generalize
many (virtual) polytopes already known
in representation theory, including the Grossberg-Karshon twisted
polytopes
and the well-known Gel'fand-Cetlin polytopes (which are special cases
of string polytopes).
Kiritchenko additionally shows that Gel'fand-Cetlin polytopes
degenerate to certain Grossberg-Karshon
twisted polytopes via a sequence of `string spaces'
\cite[\S 4]{Kiritchenko}.
As in the case of Grossberg and
Karshon, since Kiritchenko's polytopes may be virtual, her character
formula is also not necessarily `purely positive' in the sense that
some coefficients may appear with a minus sign.
Therefore, it is of
interest (see also her discussion in \cite[\S 3.3]{Kiritchenko}) to
determine when Kiritchenko's virtual polytopes are in fact `true'
polytopes. The results of this note may be viewed as a partial answer
to this more general question.

\medskip

\noindent \textbf{Acknowledgements.}
The first author is partially supported by an NSERC Discovery Grant (Individual),
an Ontario Ministry of Research
and Innovation Early Researcher Award, a Canada Research Chair
(Tier 2) award, and a Japan Society for the Promotion of Science
Invitation Fellowship for Research in Japan (Fellowship ID
L-13517). Both authors thank the Osaka City
University Advanced Mathematics Institute for its hospitality while
part of this research was conducted.


\section{Background}\label{sec:background}

\subsection{Twisted cubes}\label{subsec:twisted cubes}
We begin by recalling the definition of \textbf{twisted cubes} given
by Grossberg and Karshon \cite[\S 2.5]{Grossberg-Karshon}. Fix a
positive integer $n$. A twisted
cube is a pair $(C(\mfc, \mfl), \rho)$ where $C(\mfc, \mfl)$
is a subset of $\R^n$ and $\rho: \R^n \to \R$ is a density function
with support precisely $C(\mfc, \mfl)$. Here $\mfc =
\{c_{ij}\}_{1 \leq i < j \leq n}$ is a collection of integers and
$\mfl = \{\ell_1, \ell_2, \ldots, \ell_n\}$ is a collection of real
numbers. In order to simplify the notation in what follows, we define
the following functions on $\R^n$ using the usual coordinates $x = (x_1, \ldots, x_n)$:
\begin{equation}\label{eq:def-A}
\begin{split}
A_n(x) = A_n(x_1, \ldots, x_n) & = \ell_n \\
A_j(x) = A_j(x_1, \ldots, x_n) & = \ell_j - \sum_{k = j+1}^n c_{jk} x_k
\textup{ for all } 1 \leq j \leq n-1.
\end{split}
\end{equation}
Notice that $A_n(x)$ is in fact a constant function and that $A_j(x)$
for $1 \leq j \leq n-1$ is a linear function depending only of the
variables $x_{j+1}, \ldots, x_n$. In order to emphasize this, we will
sometimes write $A_j(x) = A_j(x_{j+1}, \ldots, x_n)$.
We also define a function $\textup{sgn}: \R \to \{\pm 1\}$ by
$\textup{sgn}(x) = 1$ for $x<0$ and $\textup{sgn}(x) = -1$ for $x
\geq 0$.
We now give a sequence of $n$ logical statements, each of which involves a logical ``or''. Specifically,
for $k$ with $1 \leq k\leq n$ we
say (S-k) is the statement
\[
\text{(S-k)}: ``A_k(x) = A_k(x_{k+1}, \ldots, x_n) < x_k < 0 \textup{ or }
0 \leq x_k \leq A_k(x) = A_k(x_{k+1}, \ldots, x_n).''
\]
Note that the condition (S-k) depends only on the last $n-k+1$ variables
 $(x_k, \ldots, x_n)$.

With the above in place we can give the definition.

\begin{definition}\label{definition:twisted cube}
With notation as above, let $C(\mfc, \mfl)$ denote
the following subset of $\R^n$:
\begin{equation}
  \label{eq:defC}
  C(\mfc, \mfl) := \{ x = (x_1, \ldots, x_n) \in \R^n \mid x \textup{ satisfies condition (S-k) 
    for all } 1 \leq k \leq n \} \subseteq \R^n.
\end{equation}
Moreover we define a density function $\rho: \R^n \to \R$ by
\begin{equation}
  \label{eq:def-rho}
  \rho(x) =
  \begin{cases}
    (-1)^n \prod_{k=1}^n \textup{sgn}(x_k) & \textup{ if } x \in
    C(\mfc, \mfl) \\
   0 & \textup{ else.}
  \end{cases}
\end{equation}
We call the pair
$(C(\mfc, \mfl), \rho)$ the \textbf{twisted cube associated to $\mfc$
  and $\mfl$.} To simplify notation we sometimes write $C$ instead of $C(\mfc, \mfl)$.
\end{definition}

\begin{remark}${}$
  \begin{itemize}
  \item Our definition is slightly different from the one given in \cite[\S
    2.5]{Grossberg-Karshon} but the two definitions may be identified
via the coordinate change $(x_1, \ldots, x_n) \mapsto (-x_1, \ldots, -x_n)$.
  \item Grossberg and Karshon refer to the
    ``standard'' twisted cube, but we will omit the word ``standard''
    from our terminology.
  \item In the representation-theoretic applications we have in mind (see Section~\ref{sec:G/B}),
    the constants $\ell_j$ are integers.
\end{itemize}
\end{remark}

As the following examples show, a twisted cube may not be a cube in
the standard sense. In particular, the set $C = C(\mfc, \mfl)$ may be neither convex
nor closed.

\begin{example}
  Let $n=2$ and let $\mfl = (\ell_1 = 3, \ell_2 = 5)$ and $\mfc =
  \{c_{12} = 1\}$. Then
\[
C = \{ (x_1, x_2) \in \R^2 \, \mid \, 0 \leq x_2 \leq 5 \textup{ and }
( 3 - x_2 < x_1 < 0 \textup{ or } 0 \leq x_1 \leq 3 - x_2 ) \}.
\]
See Figure~\ref{figure:first example}. The value of the density function $\rho$ is recorded within each region.

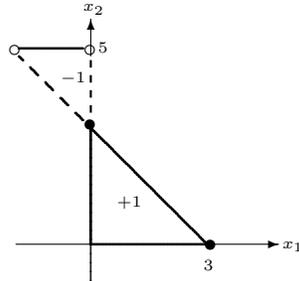
\setlength{\unitlength}{0.5cm}
\begin{figure}[h]
\centering
\begin{minipage}{1\textwidth}
\centering
\begin{picture}(4,6)(0,0)
\put(-2,0){\vector(1,0){7}}
\put(5.1, -0.1){\tiny{$x_1$}}
\curve(0,-1,0,0)
\put(0,5){\vector(0,1){1}}
\put(-0.2,6.2){\tiny{$x_2$}}
\put(3,-0.2){$\bullet$}\put(3,-0.7){\tiny{$3$}}
\put(-0.2,5){$\circ$}\put(0.2,5.1){\tiny{$5$}}
\put(-0.2,3){$\bullet$}
\put(-2.2,5){$\circ$}
\thicklines
\curve(-1.9,5.2,-0.2,5.2)
\curve(0,3.1,3.1,0)
\curve(0,0,3.1,0)
\curve(0,3.1,0,0)
\multiput(0,3.1)(0,0.4){5}{\line(0,1){0.2}}
\curve(-0.2,3.3,-0.4,3.5)\curve(-0.6,3.7,-0.8,3.9)
\curve(-1,4.1,-1.2,4.3)\curve(-1.4,4.5,-1.6,4.7)
\curve(-1.8,4.9,-1.9,5)
\put(0.7,1){\tiny{$+1$}}
\put(-0.8,4.3){\tiny{$-1$}}

\end{picture}

\end{minipage}

\label{figure:first example}
\end{figure}

Note in particular that $C$ does \emph{not} contain the points $\{ (0, x_2) \mid 3 <
x_2 < 5 \}$ and the points $\{ (x_1, x_2) \, \mid \, 3 < x_2 < 5
\textup{ and } x_1 = 3 - x_2 \}$, so $C$ is not closed, and it is
also not convex.
\end{example}

\begin{example}\label{example:second example}
  Let $n=2$ and let $\mfl = (\ell_1 = -7, \ell_2 = 5)$ and $\mfc =
  \{c_{12} = -1\}$. Then
\[
C = \{ (x_1, x_2) \in \R^2 \, \mid \, 0 \leq x_2 \leq 5 \textup{ and }
( -7 + x_2 < x_1 < 0 \textup{ or } 0 \leq x_1 \leq -7+ x_2 ) \}.
\]
See Figure~\ref{figure:second example}. Here $C$ is convex but not closed.

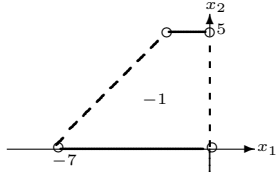
\setlength{\unitlength}{0.3cm}
\begin{figure}[h]
\centering
\begin{minipage}{1\textwidth}
\centering
\begin{picture}(4,6)(0,0)
\put(-9,0){\vector(1,0){11}}
\put(2.1, -0.1){\tiny{$x_1$}}
\curve(0,-1,0,0)
\put(0,5){\vector(0,1){1}}
\put(-0.2,6.2){\tiny{$x_2$}}
\put(-7,-0.2){$\circ$}\put(-7,-0.7){\tiny{$-7$}}
\put(0.3,5.1){\tiny{$5$}}
\put(-0.3,4.9){$\circ$}
\put(-2.2,4.9){$\circ$}
\put(-0.2,-0.2){$\circ$}
\thicklines
\curve(-0.2,5.2,-1.8,5.2)
\curve(-0.3,0,-6.6,0)

\multiput(0,0.1)(0,0.7){7}{\line(0,1){0.3}}
\curve(-6.8,0.2,-6.5,0.5)
\curve(-6.2,0.9,-5.9,1.2)
\curve(-5.6,1.5,-5.3,1.8)
\curve(-5,2.1,-4.7,2.4)
\curve(-4.4, 2.7, -4.1,3)
\curve(-3.8,3.3,-3.5,3.6)
\curve(-3.2,3.9,-2.9,4.2)
\curve(-2.6,4.5,-2.3,4.8)

\put(-3,2){\tiny{$-1$}}

\end{picture}

\end{minipage}

\label{figure:second example}
\end{figure}

\end{example}

The goal of the present manuscript is to give conditions under which
the Grossberg-Karshon twisted cube associated to $\mfl$ and $\mfc$ is
``untwisted'', i.e., the support $C = C(\mfc, \mfl)$ is a  convex
polytope (in particular $C$ is closed), and the density function
$\rho$ is constant and equal to $1$ on $C$ and $0$ elsewhere. In this
direction it is useful to note that, roughly speaking, the ``twistedness'' of $(C,\rho)$
is a consequence of the logical ``or'' present in the statements (S-k).
If it happens that the constants $\mfc =
\{c_{ij}\}$ and $\mfl = \{\ell_1, \ldots, \ell_n\}$ are such that the
only way to satisfy~\eqref{eq:defC} is to always satisfy the
inequality in the second statement of (S-k) (i.e. $0 \leq x_k \leq A_k(x)$), then $C
= C(\mfc, \mfl)$ is an intersection of closed half-spaces and is a
closed convex polytope. The definition below formalizes this simple
idea.

\begin{definition}
  Let $n, \mfc, \mfl$ and $A_j$ be as above. We say that $\mfc, \mfl$
  \textbf{satisfy condition (P)} if
  \begin{enumerate}
  \item[(P-n)] $\ell_n \geq 0$
  \end{enumerate}
  and for every integer $k$ with $1 \leq k \leq n-1$, the following statement, which we refer to as condition (P-k),
    holds:
\begin{enumerate}
\item[(P-k)]
      if $(x_{k+1}, \ldots, x_n)$ satisfies
     \begin{equation*}
      \begin{split}
      0 \leq x_n \leq A_n = \ell_n \\
      0 \leq x_{n-1} \leq A_{n-1}(x_n) \\
     \vdots \\
     0 \leq x_{k+1} \leq A_{k+1}(x_{k+2}, \ldots, x_n) \\
       \end{split}
       \end{equation*}
   then $A_k(x_{k+1},\ldots, x_n) \geq 0$.
 \end{enumerate}
In particular, condition (P) holds if and only if the conditions (P-1) through (P-n) all hold.
\end{definition}

It is not difficult to see that if $\mfc, \mfl$ satisfy condition (P)
then $C = C(\mfc, \mfl)$ is the closed convex polytope defined by
\begin{equation}\label{eq:PD}
\{x \in \R^n \, \mid \, 0 \leq x_n \leq A_n=\ell_n, 0 \leq x_{n-1}
\leq A_{n-1}(x_n), \cdots, 0 \leq x_1 \leq A_1(x_2, \ldots, x_n) \}.
\end{equation}

In what follows, we frequently find it useful
 to argue inductively on the size of $n$. For this purpose it is
 convenient to define the following.

\begin{definition}\label{def:C(k)}
For a fixed $n, \mfc, \mfl$ and fixed
 $1 \leq k \leq n$, we define a subset $C(k)$ of $\R^k$ defined
 to be the set of points satisfying conditions (S-n), (S-(n-1)), ...,
 S(n-k+1). Here we identify $\R^k$ with the subspace of
 $\R^n$ corresponding to the variables $(x_{n-k+1}, \ldots, x_n)$.
\end{definition}

The following is immediate.

\begin{lemma}\label{lemma:C(k)}
  Let $n$ be a positive integer and $\mfc, \mfl$ fixed constants. For
  any $k$ with $1 \leq k \leq n$, the projection $\pi_k: \R^n
  \to \R^k$ given by $(x_1, \ldots, x_n) \mapsto (x_{n-k+1},
  \ldots, x_n)$ induces a surjective map $C = C(\mfc, \mfl) \to
  C(k)$. In particular, the image $\pi_k(C)$ is precisely equal
  to $C(k)$.
\end{lemma}

\subsection{Divisors, Cartier data, and polytopes}\label{subsec:divisors}

In this section we briefly recall a construction given by
Pasquier \cite{Pasquier} of a non-singular toric variety $X(\mfc)$ and a
torus-invariant divisor $D(\mfc, \mfl)$ on $X(\mfc)$, associated to
the data $\mfc = \{c_{ij}\}$ and $\mfl = \{\ell_j\}$.
As mentioned in the introduction,
the variety
$X(\mfc)$ arises as the special fiber of a toric degeneration of a
Bott-Samelson variety. Indeed, Pasquier's construction of this toric degeneration is the algebro-geometric version of the
degeneration of complex structures given by Grossberg and
Karshon in \cite{Grossberg-Karshon}. However, this
perspective is not necessary in this paper; here we only need the
explicit formula for the relevant toric-invariant divisor $D(\mfc,
\mfl)$. (A detailed discussion of the relation between $D(\mfc, \mfl)$
and the toric degeneration mentioned above is in \cite{Harada-Yang}.)

Let $n$ be a fixed positive integer and $\mfc = \{c_{ij}\}$ and $\mfl = \{\ell_k\}$ be collections of constants as in Section~\ref{sec:background}. In this section we assume the $\ell_k$ are integers for all $k$.
Let $\{e_1^+, \ldots, e_n^+\}$ denote the standard
basis of $\R^n$ and let $\Z^n$ be the standard lattice in $\R^n$
generated by the $e_j^+$. Define vectors $e_j^-$ for $j = 1, \ldots,
n$ by the formula
\begin{equation}
  \label{eq:def-ejminus}
  e_j^- := - e_j^+ - \sum_{k>j} c_{jk} e_k^+.
\end{equation}
Note that the $e_j^-$ depend only on $\mfc = \{c_{ij}\}$ and not on $\mfl$.

For basic background on toric varieties we refer the reader to \cite{Cox-Little-Schenck}.
Let $\Sigma_{\mfc}$ denote the fan consisting of cones generated by
subsets of $\{e_1^+, \ldots, e_n^+, e_1^-, \ldots, e_n^-\}$ which do
not contain any subsets of the form $\{e_j^+, e_j^-\}$. From this it
easily follows that the set of maximal ($n$-dimensional) cones
$\Sigma_{\mfc}(n)$ of $\Sigma_{\mfc}$ is in 1-1 correspondence with
the vectors $\{\sigma = (\sigma_1, \ldots, \sigma_n) \, \mid \, \sigma_i
\in \{+,-\} \}$ via the map $\sigma \mapsto \mr{Cone} \{e_1^{\sigma_1}, \ldots,
e_n^{\sigma_n}\}$. In particular, $\lvert \Sigma_{\mfc}(n) \rvert =
2^n$.
Let $X(\mfc) = X(\Sigma_{\mfc})$ denote the toric variety associated
to the fan $\Sigma_{\mfc}$.
Since each cone is generated by a $\Z$-basis, the fan $\Sigma_{\mfc}$ is smooth and thus $X(\mfc)$ is smooth \cite[Theorem 3.1.19]{Cox-Little-Schenck}.
Next, for each $j$, $1 \leq j \leq n$, let $D_{e_j^-}$ be the
torus-invariant divisor on $X(\mfc)$ corresponding to the ray spanned
by $e_j^-$ \cite[\S 4.1]{Cox-Little-Schenck}. We now define the
torus-invariant divisor $D(\mfc, \mfl)$ on $X(\mfc)$ by
\begin{equation}
  \label{eq:def-Dcl}
  D(\mfc, \mfl) := \sum_{j=1}^n \ell_j D_{e_j^-}.
\end{equation}
Since $X(\mfc)$ is non-singular, the divisor $D(\mfc, \mfl)$ is
Cartier \cite[Theorem 4.0.22]{Cox-Little-Schenck}. In what follows we give
an explicit
computation of the so-called Cartier data
$\{m_{\sigma}\}_{\sigma \in
  \Sigma_{\mfc}(n)}$ of $D(\mfc, \mfl)$
\cite[Theorem 4.2.8]{Cox-Little-Schenck}. We first
set some notation. Here we view $\R^n$ as $N \otimes \R$ where $N \cong
\Z^n$ is the free abelian group of one-parameter subgroups of a torus
$(\C^*)^n$. Let $M \cong \Z^n$ denote the free abelian group of
characters of the same torus and let $\langle \cdot, \cdot \rangle: M
\times N \to \Z$ denote the usual bilinear pairing between characters
and one-parameter subgroups \cite[\S 1.1]{Cox-Little-Schenck}. (With
the standard identifications of $M$ and $N$ with $\Z^n$, the bilinear
pairing above is just the usual inner product.)
For any fan $\Sigma$, let $\Sigma(1)$ denote the set of
$1$-dimensional cones in $\Sigma$, and for $\rho \in \Sigma(1)$ let
$D_\rho$ denote the torus-invariant divisor corresponding to $\rho$
\cite[\S 4.1]{Cox-Little-Schenck}. Further, we let $u_\rho$ be the
minimal generator of $N \cap \rho$. Let $\Sigma_{max}$ denote the set
of maximal cones, and for $\sigma \in \Sigma_{max}$ let $\sigma(1)$
denote the set of rays in $\sigma$.

We quote the following from \cite{Cox-Little-Schenck}.

 \begin{theorem}\label{theorem:Cartier} (cf. \cite[Theorem 4.2.8]{Cox-Little-Schenck})
    Let $\Sigma$ be a fan, $X(\Sigma)$ the toric variety associated to
    $\Sigma$, and $D = \sum_{\rho} a_\rho D_\rho$ a divisor. Then the
    following are equivalent:
    \begin{itemize}
    \item $D$ is Cartier.
    \item For each $\sigma \in \Sigma_{\max}$, there exists an element $m_\sigma \in
      M$ with $\langle
      m_\sigma, u_\rho \rangle = - a_\rho$ for all $\rho \in
      \sigma(1)$.
    \end{itemize}
    A collection $\{m_\sigma\}$ satisfying the above
    conditions is called the \textbf{Cartier data of $D(\mfc,
      \mfl)$}.
  \end{theorem}

In our case, the maximal cones are precisely the
$n$-dimensional cones $\Sigma_{\mfc}(n)$, which in turn are in 1-1
correspondence with the set of $\sigma$ in $\{+,-\}^n$. Using this
identification we may refer to $\sigma = (\sigma_1, \ldots,
\sigma_n) \in \{+,-\}^n$ as a maximal cone. The following lemma gives
an explicit computation of the Cartier data $\{m_\sigma\}$ of $D(\mfc,
\mfl)$.

\begin{lemma}\label{lemma:msigma formula}
  Let $\sigma = (\sigma_1, \ldots, \sigma_n) \in \{+,-\}^n$ be a
  maximal cone in $\Sigma_{\mfc}(n)$. Then the associated Cartier data
  $m_{\sigma} = (m_{\sigma,1}, \ldots, m_{\sigma, n}) \in \Z^n$ of
  $D(\mfc, \mfl)$ is
    given by the formula
    \begin{equation}
      \label{eq:def-msigma-j}
      m_{\sigma, j}  =
      \begin{cases}
        0 & \textup{ if } \sigma_j = + \\
        A_j(m_{\sigma, j+1}, \ldots, m_{\sigma, n}) & \textup{ if }
        \sigma_j = -
      \end{cases}
    \end{equation}
for all $1 \leq j \leq n$.
\end{lemma}

\begin{proof}
By Theorem~\ref{theorem:Cartier}, the vector $m_\sigma \in \Z^n$
must satisfy the condition
\[
\langle m_\sigma, u_\rho \rangle = - a_\rho
\]
for all $1$-dimensional rays $\sigma(1)$ in $\sigma$. Recall that
$\sigma = (\sigma_1, \ldots, \sigma_n) \in \{+,-\}^n$ is identified
with $\mr{Cone}\{e_1^{\sigma_1}, \ldots, e_n^{\sigma_n}\}$ and
that $D(\mfc, \mfl) = \sum_{j=1}^n \ell_j D_{e_j^-}$. Thus $m_\sigma$
is required to satisfy the $n$ equations
\begin{equation}
  \langle m_\sigma, e_j^{\sigma_j} \rangle =
  \begin{cases}
    0  & \textup{ if } \sigma_j = + \\
   -\ell_j & \textup{ if } \sigma_j = -.\\
  \end{cases}
\end{equation}
In particular, since $e_j^+$ is a standard basis vector, we see that
if $\sigma_j = +$ then $m_{\sigma,j} = 0$. If $\sigma_j = -$, then
from the definition of the $e_j^-$ in~\eqref{eq:def-ejminus} it follows
that $m_{\sigma, j} = A_j(m_{\sigma, j+1}, \ldots, m_{\sigma, n})$.

\end{proof}

Next recall from \cite[\S 6.1]{Cox-Little-Schenck} that there
exists a polytope $P_D$ associated to any torus-invariant
Cartier divisor $D = \sum_\rho a_\rho D_\rho$ on a toric
variety. Specifically, we define (cf. \cite[Equation
(6.1.1)]{Cox-Little-Schenck}):
\begin{equation}
  \label{eq:def-PD-general}
  P_D := \{ m \in M \otimes \R \, \mid \, \langle m, u_\rho \rangle \geq -
  a_\rho \textup{ for all } \rho \in \Sigma(1) \}.
\end{equation}
In our setting, the $1$-dimensional cones are those
spanned by  $n$ primitive vectors $\{e_1^{\si_1},\dots,e_n^{\si_n}\}$,
so we obtain
\begin{equation}
  \label{eq:def-PD-ourcase}
  P_{D(\mfc, \mfl)} = \{ x \in M \otimes \R \cong \R^n \, \mid \, 0 \leq x_j \leq
  A_j(x) \textup{ for all } 1 \leq j \leq n \}
\end{equation}
by using the definition of the $e_j^- (j=1,\dots,n)$. Here
we take the usual identification of $M \otimes \R$
with $\R^n$.

\section{Untwistedness of twisted cubes and base-point free divisors}\label{sec:Untwistedness and Base-point free divisor}

In this section we prove the main result
(Theorem~\ref{theorem:Untwistedness and Base-point free divisor}) of
this note, which states that the Grossberg-Karshon twisted cube is
``untwisted'' (i.e. the support $C(\mfc, \mfl)$ is a closed, convex
polytope and the support function is constant and equal to $1$ on
$C(\mfc, \mfl)$, cf. Definition~\ref{definition:untwisted}) if and only if the
divisor $D(\mfc, \mfl)$ of Section~\ref{subsec:divisors} is
basepoint-free. Thus we are able to completely characterize the
untwistedness in terms of the geometry of an associated toric
variety. As a consequence, we also obtain a computationally convenient
method for determining whether or not the twisted cube is untwisted
(cf. Remark~\ref{remark:computations}).

Fix a positive integer $n$. Let
$\mfc=\{c_{ij}\}_{1\le i<j\le n}$ and $\mfl=\{\ell_1,\dots,\ell_n\}$
be fixed integers. The following proposition gives several equivalent
conditions for the untwistedness (cf. Definition~\ref{definition:untwisted})
of the Grossberg-Karshon
twisted cube $(C = C(\mfc, \mfl), \rho)$ and additionally gives an
explicit computation of $C$ in terms of toric geometry, namely, it is
the polytope $P_{D(\mfc, \mfl)}$. We note also that the
proposition shows that if $C$ is closed then it is
automatically convex. The conditions (b) and (c) in the proposition,
involving the Cartier data $\{m_\sigma\}$, are the most useful for
later applications.

\begin{proposition}\label{proposition:Untwistedness and Base-point free divisor}

 Let $\mfc, \mfl$ be as above. Let $(C(\mfc, \mfl),\rho)$ denote the
  corresponding Grossberg-Karshon twisted polytope.
Let
$m_\sigma$ be the Cartier data of the divisor $D(\mfc, \mfl)$ on $X(\mfc)$
and let $P_{D(\mfc, \mfl)}$ denote the associated polytope
given in~\eqref{eq:def-PD-ourcase}.
Then the following are equivalent:
  \begin{enumerate}
  \item[(a)] $C = C(\mfc,\mfl)$ is closed (as a subset of $\R^n$ with the
    usual Euclidean topology).
  \item[(b)] $m_\sigma \in C$ for all $\sigma$.
  \item[(c)] $m_{\sigma, k} \geq 0$ for all $\sigma$ and all $k$.
  \item[(d)] $\mfc$ and $\mfl$ satisfy the condition (P).
  \item[(e)] $C = P_{D(\mfc, \mfl)}$.

  \end{enumerate}

\end{proposition}

\begin{definition}\label{definition:untwisted}
  If any of the above (equivalent) conditions hold, we say that the
Grossberg-Karshon twisted cube $(C=C(\mfc, \mfl), \rho)$ is
\textbf{untwisted}. Equivalently, from condition (e) above it follows
that $(C,\rho)$ is untwisted exactly when $C$ is a closed convex
polytope (namely $P_{D(\mfc,\mfl)}$) and the support for $\rho$ is
constant and equal to $1$ on $C$ (since $P_{D(\mfc,\mfl)}$ lies in the
positive orthant by definition).
\end{definition}

In order to prove Proposition ~\ref{proposition:Untwistedness and Base-point free divisor} we need a preliminary lemma.

\begin{lemma}\label{lemma:closure}
  For all $\sigma \in \{+,-\}^n$, the integer vector $m_\sigma$ lies
  in the closure $\overline{C}$ of $C$ (with respect to the usual
  Euclidean topology).
\end{lemma}

\begin{proof}
  We use induction on $n$. For the base case $n=1$, we have
\[
C = \{x \in \R \, \mid \, A_1 = \ell_1 < x_1 < 0 \textup{ or } 0 \leq
x_1 \leq A_1 = \ell_1\}.
\]
Since $\ell_1$ is a fixed constant, we have that either $\ell_1 \geq
0$ or $\ell_1 < 0$. In the case $\ell_1 \geq 0$
we have $C = \{0 \leq x_1 \leq \ell_1\}$. For $\sigma = +$ we have
by~\eqref{eq:def-msigma-j} that $m_\sigma = 0$ and for $\sigma = -$ we
have $m_\sigma = A_1 =\ell_1$. In both cases $m_\sigma \in C$ as
desired. On the other hand, if $\ell_1 < 0$ then $C = \{\ell_1 < x_1 <
0\}$ and $\overline{C} = \{ \ell_1 \leq x_1 \leq 0\}$. A similar argument
shows $m_\sigma \in \overline{C}$ for both $\sigma = +$ and $\sigma=-$.

Now suppose the lemma holds for $n-1$.
More specifically, we assume that for any choice of  vector
$\sigma' = (\sigma_2, \ldots, \sigma_n) \in \{+,-\}^{n-1}$, the vector
$m_{\sigma'}$ lies in the closure of $C(n-1)$ where $C(n-1)$ is the
polytope defined in Definition~\ref{def:C(k)}. Now suppose $\sigma =
(\sigma_1, \sigma_2, \ldots, \sigma_n) \in \{+,-\}^n$ and let
$m_\sigma$ be the associated Cartier data defined
by~\eqref{eq:def-msigma-j}. We wish to show that $m_\sigma$ lies in the
closure of $C$, i.e. for arbitrary fixed $\varepsilon > 0$ we wish to
find $x = (x_1, \ldots, x_n) \in C$ with $\lvert m_\sigma - x \rvert <
\varepsilon$, where $\lvert \cdot \rvert$ denotes the usual Euclidean
norm in $\R^n$. Let $\pi_{n-1}: \R^n \to \R^{n-1}$
denote the projection $(x_1, \ldots, x_n)
\mapsto (x_2, \ldots, x_n)$ as in Lemma~\ref{lemma:C(k)}. Note that
$\pi_{n-1}(m_\sigma) =
m_{\sigma'}$.

We first consider the case $\sigma_1 = -$, so $m_{\sigma,1} =
A_1(m_{\sigma,2}, \ldots, m_{\sigma,n}) = A_1(m_{\sigma'})$. Since the
function $A_1(x_2, \ldots, x_n)$ is continuous, for arbitrary
$\varepsilon>0$ we know there exists $\delta > 0$ such that for any
$x' = (x_2, \ldots, x_n)$ with $\lvert m_{\sigma'} - x' \rvert <
\delta$ then $\lvert A_1(m_{\sigma'}) - A_1(x') \rvert  = \lvert
m_{\sigma,1} - A_1(x') \rvert <
\frac{\varepsilon}{3}$. Let $\varepsilon' =
\min\{\frac{\varepsilon}{3}, \delta\}$. By the inductive assumption we
know there exists a point $x' = (x_2, \ldots, x_n) \in C(n-1)$ such
that $\lvert m_{\sigma'} - x' \rvert < \varepsilon'$. Moreover by
definition of $C$ we know $x = (x_1,\ldots, x_n) \in C$ (so
$\pi_{n-1}(x) = x'$) exactly if $A_1(x_2, \ldots, x_n) < x_1 < 0$ or
$0 \leq x_1 \leq A_1(x_2, \ldots, x_n)$. In particular there exists
$x\in C$ with $\lvert x_1 - A_1(x_2, \ldots, x_n) \rvert <
\frac{\varepsilon}{3}$. By the triangle inequality we have
\begin{equation}
\begin{split}
\lvert x - m_\sigma \rvert & \leq \lvert x_1 - m_{\sigma,1} \rvert +
\lvert x' - m_{\sigma'} \rvert \\
 & \leq \lvert x_1 - A_1(x_2, \ldots, x_n) \rvert + \lvert A_1(x_2, \ldots,
 x_n) - m_{\sigma,1} \rvert + \lvert x' - m_{\sigma'} \rvert \\
 & \leq \frac{\varepsilon}{3} + \frac{\varepsilon}{3} +
 \frac{\varepsilon}{3} = \varepsilon,
\end{split}
\end{equation}
as desired.

Next we consider the case when $\sigma_1 = +$, so $m_{\sigma,1} =
0$. From the inductive assumption we know that there exists $x' =
(x_2, \ldots, x_n) \in C(n-1)$ with $\lvert m_{\sigma'} - x' \rvert <
\frac{\varepsilon}{2}$. Again by definition of $C$ we know $x = (x_1,
\ldots, x_n)$ is in $C$ exactly if $A_1(x_2, \ldots, x_n) < x_1 < 0$
or $0 \leq x_1 \leq A_1(x_2, \ldots, x_n)$. In particular there exists
$x \in C$ with $\lvert x_1 \rvert < \frac{\varepsilon}{2}$. By the
triangle inequality we have
\begin{equation}
 \begin{split}
\lvert x - m_{\sigma} \rvert & \leq \lvert x_1 - m_{\sigma,1} \rvert +
\lvert x' - m_{\sigma'} \rvert \\
 & = \lvert x_1 \rvert + \lvert x' - m_{\sigma'} \rvert \\
 & \leq \frac{\varepsilon}{2}    + \frac{\varepsilon}{2}    =
 \varepsilon,
  \end{split}
\end{equation}
as desired.
This holds for arbitrary $\varepsilon > 0$ so we conclude in both
cases that $m_\sigma \in \overline{C}$ as was to be shown.
\end{proof}

\begin{proof}[Proof of Proposition~\ref{proposition:Untwistedness and Base-point free divisor}]
The proof is by induction on the size of $n$. More specifically,
we wil prove first that if $n=1$ then all the statements (a) through
(e) are equivalent. Then, assuming the equivalences for $n-1$, we prove
the equivalences for $n$.

First suppose $n=1$. In this case $\mfc = \{c_{ij}\} = \emptyset$ and $\mfl =
\{\ell_1\}$. From the definition of $C = C(\mfc, \mfl)$ we have
\[
C = \{ x \in \R : \ell_1 < x < 0 \textup{ or } 0 \leq x \leq \ell_1
\}.
\]
In particular, it is immediate that $C$ is closed if and only if
$\ell_1 \geq 0$.

(a) $\Leftrightarrow$ (b): Since $\ell_1 \geq 0$,
from the formula for $m_\sigma$ it follows that $m_{-}
= \ell_1$ and $m_+ = 0$ are both in $C$. Moreover, $m_{-}$ is in $C$
if and only if $\ell_1 \geq 0$.

(a) $\Leftrightarrow$ (c): Since $\ell_1 \geq 0$ in this case, again
from the definition of the $m_\sigma$, we
see that $m_\sigma \geq 0$ for all $\sigma$. Moreover if $m_{-} \geq
0$ then $\ell_1 \geq 0$.

(a) $\Leftrightarrow$ (d): Condition (P) in this case is precisely
that $\ell_1 \geq 0$.

(a) $\Leftrightarrow$ (e): In this case $P_{D(\mfc, \mfl)} = \{x \in \R: 0 \leq x
\leq \ell_1\}$, which is equal to $C$ exactly when $\ell_1 \geq 0$.

\medskip
Assume by induction that the statements (a) through (e) are
equivalent for $n-1$. We now do the rounds for $n$.

\medskip

(a) $\Rightarrow$ (b). By Lemma~\ref{lemma:closure}, $m_\sigma \in \overline{C}$ for all $\sigma$. Since $C$ is closed, $\overline{C} = C$ and $m_\sigma \in C$ for all $\sigma$ as desired.

(b) $\Rightarrow$ (c). Suppose for a contradiction that there exists $\sigma$ and $k$ with $m_{\sigma,k} < 0$. By definition of $m_\sigma$ this means $\sigma_k = -$ and $m_{\sigma, k} = A_k(m_{\sigma, k+1}, \ldots, m_{\sigma,n}) < 0$. But then $m_{\sigma,k}$ does not satisfy condition (S-k), since
\[
A_k(m_{\sigma, k+1}, \ldots, m_{\sigma,n}) \not < m_{\sigma,k}.
\]
Thus by definition of $C$ this implies $m_\sigma \not \in C$, contradicting (b).

(c) $\Rightarrow$ (d). By induction we already know that conditions
(P-2) through (P-n) hold, so it suffices to check the condition
(P-1). Since we assume $m_{\sigma, k} \geq 0$ for all $\sigma$ and
$k$, in particular this implies that $m_{\sigma', k} \geq 0$ for all
$\sigma' \in \{+,-\}^{n-1}$ and $2 \leq k \leq n$. By our inductive
assumption we then know that $C(2) = P_{D'}$ where $D'$ is the divisor
determined by $\mfc' = \{c_{ij}\}_{2 \leq i<j\leq n}$ and $\mfl' =
\{\ell_2, \ldots, \ell_n\}$. Moreover $C(2)$ is the convex hull of the
$m_{\sigma'}$ as $\sigma'$ ranges over $\{+,-\}^{n-1}$.
In particular, in this case the condition (P-1) is
equivalent to saying that
if $(x_2, \ldots, x_n) \in P_{D'}$ then $A_1(x_2, \ldots, x_n) \geq
0$. By the above, this is equivalent to the statement that if $(x_2,
\ldots, x_n)$ is in the convex hull of the $\{m_{\sigma'}\}$, then
$A_1(x_2, \ldots, x_n) \geq 0$. Since the function $A_1$ is linear, to
check this condition it suffices to check at the vertices
$\{m_{\sigma'}\}$. For any $\sigma' \in \{+,-\}^{n-1}$ consider
$\sigma = (-, \sigma'_2, \ldots, \sigma'_n)$.
Then $m_{\sigma,1} = A_1(m_{\sigma'})$, but
by assumption this is $\geq 0$. Hence (P-1) is satisfied, as desired.

(d) $\Rightarrow$ (e).  This is immediate from the definition of $C$
and $P_{D(\mfc, \mfl)}$.

(e) $\Rightarrow$ (a).   Since $P_{D(\mfc, \mfl)}$ is a closed convex polytope by
definition, if $C=P_{D(\mfc, \mfl)}$ then in particular $C$ is closed.

\medskip
The above shows that the conditions (a) through (e) are equivalent for
any $n$.
\end{proof}

The proposition above gives us many equivalent ways to check the
untwistedness of the Grossberg-Karshon twisted cube, and includes some
characterizations involving the Cartier data of $D(\mfc, \mfl)$ on
$X(\mfc)$. These allow us to prove our main result, which connects the
basepoint-freeness of the divisor $D(\mfc, \mfl)$ with the
untwistedness of the twisted cube.

\begin{theorem}\label{theorem:Untwistedness and Base-point free
    divisor}
  In the setting of Proposition
  ~\ref{proposition:Untwistedness and Base-point free divisor}, the
  Grossberg-Karshon twisted cube
  $(C(\mfc,\mfl), \rho)$ is untwisted if and only if
the divisor $D(\mfc, \mfl)$ on $X(\mfc)$ is
basepoint-free. Equivalently, the twisted cube is untwisted if and
only if
$m_\sigma \in P_{D(\mfc, \mfl)}$ for all $\sigma$ if and only if
$P_{D(\mfc, \mfl)}$ is the convex hull of the $\{m_\sigma\}_{\sigma \in \{+,-\}^n}$.
\end{theorem}

\begin{proof}
It is well-known that $D(\mfc, \mfl)$ is basepoint-free if and only if
$m_\sigma \in P_{D(\mfc,\mfl)}$ for all $\sigma$ if and only if
$P_{D(\mfc, \mfl)}$ is the convex hull of the $\{m_{\sigma}\}_{\sigma
  \in \{+,-\}^n}$
\cite[Theorem 6.1.7]{Cox-Little-Schenck}.
If $m_\sigma \in P_{D(\mfc,\mfl)}$ for all $\sigma$, then since $P_{D(\mfc, \mfl)}$ is by definition a subset of $C$
from Proposition ~\ref{proposition:Untwistedness and Base-point free
  divisor} we may
conclude that $(C,\rho)$ is untwisted. On the other hand, if $C$ is
untwisted then the conditions (b) and (e) of Proposition
~\ref{proposition:Untwistedness and Base-point free divisor} holds
which together these imply that $m_\sigma \in P_{D(\mfc, \mfl)}$ for all
$\sigma$ and thus $D(\mfc, \mfl)$ is base-point free.
\end{proof}

\begin{remark}\label{remark:computations}
In practice, the conditions which is useful for computations is Proposition ~\ref{proposition:Untwistedness and Base-point free divisor}(c).
 Indeed, the condition Proposition ~\ref{proposition:Untwistedness and Base-point free divisor}(c) is used in \cite{HaradaLee}
  in which the special case of twisted cubes
  arising from representation theory is studied in more detail.
\end{remark}

\begin{example}
Let $\mfc=\{c_{12}=2\}$ and $\mfl=\{\ell_1=4,\ell_2=3\}$.
The linear functions $A_1$ and $A_2$ are
\[A_2=3,\quad A_1(x_2)=4-2x_2.\] Since $A_1(3)=-2<0$, $\mfc$ and $\mfl$ do not satisfy the condition $(P)$.
In this example we also have $D(\mfc,\mfl)=4D_{e_1^{-}}+3D_{e_2^{-}}$
and $P_{D(\mfc, \mfl)} = \{(x_1,x_2)\in\R^2|  0 \leq x_2 \leq 3, 0
\leq x_1 \leq 4 - 2x_2 \}.$ For the maximal cone $\si=(-,-)$ spanned
by the ray $e_1^{-}$,  and $e_2^{-}$, we see that
$m_{\si}=(-2,3)\notin P_D$. Thus $D(\mfl, \mfc)$ is not base-point free.
\end{example}

\section{Untwistedness vs. convexity via positivity of the $\ell_k$}\label{sec:positivity}

As Example
~\ref{example:second example} shows,
 the untwistedness of the twisted cube
 is not equivalent to the convexity of the twisted cube, i.e., a
 twisted cube may be convex but not untwisted.
 In this section we study their relation via the positivity of the
 collection of integers $\ell_1,\dots,\ell_n$, which in turn relates
 to the effectivity of the divisor $D(\mfc, \mfl)$.

We begin with the following observation.

\begin{lemma}\label{lemma:positivity}
  Let $n$ be a positive integer and $\mfc = \{c_{ij}\}, \mfl =
  \{\ell_1, \ldots, \ell_n\}$ be fixed integers. If the
  Grossberg-Karshon twisted cube $(C = C(\mfc, \mfl), \rho)$ is
  untwisted, then all $\ell_i \geq 0$.
\end{lemma}

\begin{proof}
  If the Grossberg-Karshon twisted cube is untwisted, then by
  Theorem ~\ref{theorem:Untwistedness and Base-point free divisor} we know condition (P) holds. So we wish
  to show that if condition (P) holds for $\mfc$ and $\mfl$ then
  $\ell_i \geq 0$ for all $i$.

We prove this by induction on the size of $n$. Suppose $n=1$. In this
case the condition (P) consists of the single statement (P-1) that
$\ell_1 \geq 0$. Now suppose the statement holds for
$n-1$. In particular, since condition (P) includes the $n-1$ statements
(P-2) to (P-n), we know that $\ell_i \geq 0$ for all $i$ with $2 \leq
i \leq n$. In particular this means that the point $(x_2, \ldots, x_n)
= (0,0,\ldots,0)$ satisfies the hypothesis of condition (P-1), since
in this case $A_k(x_{k+1}, \ldots, x_n) = \ell_k \geq 0$ and $x_k = 0
\leq \ell_k$ for all $k$
with $2 \leq k \leq n$. Then the conclusion of (P-1) states that
$A_1(x_2, \ldots, x_n) = A_1(0,0,\ldots,0) = \ell_1$ must be
non-negative, as desired.
\end{proof}

The theorem below shows that in our setting, if the divisor $D(\mfc,
\mfl)$ is basepoint-free then it is also effective. We also prove a
partial converse.

\begin{theorem}\label{theorem:positivity}
Let $n$ be a positive integer and let $\mfc = \{c_{ij}\}$ and $\mfl =
(\ell_k)$ be fixed integers.
\begin{enumerate}
  \item[(a)] If the Grossberg-Karshon twisted cube $(C=C(\mfc, \mfl),
    \rho)$ is untwisted, then $C(\mfc, \mfl)$ is convex and $\ell_k
    \geq 0$ for all $k$.
    Equivalently, if $D(\mfc, \mfl)$ is basepoint-free,
    then it is also effective.
\item[(b)] If $C(\mfc,\mfl)$ is convex and $\ell_k$ is strictly
  positive for all $k$, then the twisted cube $(C(\mfc, \mfl),\rho)$
  is untwisted (equivalently, $D(\mfc,\mfl)$ is base-point free).
\end{enumerate}
\end{theorem}

\begin{proof}
  If the twisted cube is untwisted, then by Theorem
  ~\ref{theorem:Untwistedness and Base-point free divisor} we know
  that $C=C(\mfc,\mfl)=P_{D(\mfc,\mfl)}$ and thus $C$ is convex. Then by Lemma~\ref{lemma:positivity} above, we see all $\ell_i$ are non-negative.

Next suppose that $C$ is convex and all $\ell_k$ are strictly
positive. Suppose in order to obtain a contradiction that $(C = C(\mfc, \mfl),
\rho)$ is \emph{not} untwisted. By Theorem~\ref{theorem:Untwistedness and Base-point free divisor} this means
that $C \neq P_D$. Since by definition $P_D$ is always a subset of
$C$, this means that there exists a point $a = (a_1, \ldots, a_n)$ in
$C \setminus P_D$. Since $a \not \in P_D$, we know there exists
some $j$ such that $a_j < 0$. On the other hand, $a \in C$ so we must
have $A_j(a_{j+1}, \ldots, a_n) < a_j < 0$.

Now let $\Gamma= \{A_j(x) < x_j < 0\}$ be the intersection of the two
open half spaces $\{A_j(x)<x_j\}$ and $\{x_j<0\}$.
Also let $\Gamma' = \{0 \leq x_j \leq A_j(x)\}$ be the intersection of
the two closed half spaces $\{0\le x_j\}$ and $\{x_j \le A_j(x)\}$.
Then by definition $C\subset \Gamma\cup\Gamma'$. Note that $a \in \Gamma$.
Also note that since all $\ell_i$ are positive, the origin
$(0,0,\ldots, 0)$ is contained in $C$. Since $C$ is in addition
assumed to be convex, we know that $ta \in C$ for all $t \in
[0,1]$. In fact, since $a \in \Gamma$, we further know that $ta \in C
\cap \Gamma$ for any
$t$ with $0 < t \leq 1$. Thus
\[
A_j(ta)-ta_j<0\Leftrightarrow
\ell_j-\sum_{k>j}c_{jk}ta_k-ta_j<0\Leftrightarrow
\ell_j<t \left(a_j+\sum_{k>j}c_{jk}a_k \right).
\]
However,  since by assumption $\ell_j>0$, the last condition in the
string of equivalences above implies
that the expression $a_j+\sum_{k>j}c_{jk}a_k$ must be strictly
positive.
Thus
\[ ta\in C\cap\Gamma^- \Rightarrow t>
\frac{\ell_j}{a_j+\sum_{k>j}c_{jk}a_k}. \]
This is a contradiction, since we may choose $t$ to be arbitrarily
close to $0$.
\end{proof}

\section{Connection to flag varieties $G/B$}\label{sec:G/B}

The study of twisted polytopes initiated by Grossberg and Karshon's
work \cite{Grossberg-Karshon} was motivated by representation theory and, in
particular, the search for polytopes
whose lattice points encode the characters of representations.
In this
section we record some initial observations concerning the
untwistedness of the Grossberg-Karshon twisted cube
when the constants $\mfc = \{c_{ij}\}$ and $\mfl = \{\ell_k\}$ are
determined from certain representation-theoretic data.

Following \cite{Grossberg-Karshon}, we let $G$ be a complex semisimple
simply-connected linear algebraic group
of rank $r$ over an
algebraically closed field $\mf{k}$.
Choose a Cartan subgroup $H\subset G$, and a Borel subgroup
$B\supset H$.
Let $\al_i$ denote the simple roots, $\al_i^{\vee}$ the coroots,
and $\varpi_i$ the fundamental weights (characterized by the relation
$\langle\varpi_i,\al_j^{\vee}\rangle=\de_{ij}$). 
Recall that the simple reflections
$s_{\be}:X\rightarrow X,
\lambda\mapsto\lambda-\langle\lambda,\be^{\vee}\rangle \be$
generate the Weyl group $W$, where we let $\be$ range among the
simple roots $\al_1,\dots,\al_r$.

Fix a choice $\lambda = \lambda_1 \varpi_1 + \ldots + \lambda_r
\varpi_r$ in the weight lattice, where $\lambda_i \in \Z$. Also fix a choice
$\underline{w}=s_{\be_1}\cdots s_{\be_n}$ of  decomposition of
an element $w$ of $W$.
For such $\lambda$ and $\underline{w}$
we define constants $\mfc, \mfl$ by the formulas (cf. \cite[\S 3.7]{Grossberg-Karshon})
\begin{equation}
  \label{eq:def cij rep}
  c_{ij}=\lee \be_j,\be_i^{\vee}\ree
\end{equation}
for $1 \leq i < j \leq n$, and
\begin{equation}
  \label{eq:def ell rep}
\ell_1=\lee \lambda,\be_1^{\vee}\ree, \dots, \ell_n=\lee \lambda,\be_n^{\vee}\ree.
\end{equation}
Note that if the
$j$-th simple reflection in the given
 word decomposition
$\underline{w}$ is equal to $\alpha_i$, then $\ell_j = \la_i$.

The following simple example illustrates these definitions.

\begin{example}
  Consider $G = SL(3,\C)$ with positive roots $\{\alpha_1,
  \alpha_2\}$. Let $\lambda = 2 \varpi_1 + \varpi_2$ and
  $\underline{w} = s_{\alpha_1} s_{\alpha_2} s_{\alpha_1}$. Then
  $(\beta_1, \beta_2, \beta_3) = (\alpha_1, \alpha_2, \alpha_1)$ and we
  have
  \begin{equation}
   \begin{split}
    c_{12} & = \langle \alpha_2, \alpha_1^{\vee} \rangle = -1 \\
    c_{13} & = \langle \alpha_1, \alpha_1^{\vee} \rangle = 2 \\
    c_{23} & = \langle \alpha_1, \alpha_2^{\vee} \rangle = -1 \\
    \mfl = (\ell_1, \ell_2, \ell_3) & = (\langle \lambda, \alpha_1^{\vee} \rangle = 2, \langle \lambda, \alpha_2^{\vee} \rangle = 1, \langle \lambda, \alpha_1^{\vee} \rangle = 2).
    \\
   \end{split}
  \end{equation}
\end{example}

As mentioned in the introduction, in the setting above Grossberg and Karshon derive a Demazure-type character formula for the irreducible $G$-representation corresponding to $\lambda$, expressed as a sum over the lattice points $\Z^n \cap C(\mfc, \mfl)$ in the Grossberg-Karshon twisted cube $(C(\mfc, \mfl),\rho)$ \cite[Theorem 5 and Theorem 6]{Grossberg-Karshon}. The lattice points appear with a plus or minus sign according the density function $\rho$. Hence their formula 
is a \emph{positive} formula if $\rho$ is constant and equal to $1$
 on all
of $C(\mfc,\mfl)$. 
From the point of view of representation theory it
is therefore of interest to determine conditions on the weight
$\lambda$ and the
 word decomposition $\underline{w} =
s_{\beta_1} \cdots s_{\beta_n}$ such that the associated
Grossberg-Karshon twisted cube is in fact untwisted. The main result
of this section, Theorem~\ref{theorem:necessary conditions} below,
gives 2 simple criteria which are necessary for untwistedness. A more
detailed analysis of this situation, which gives both necessary and sufficient conditions,
is in \cite{HaradaLee}.

\begin{theorem}\label{theorem:necessary conditions}

Let $\lambda$ and $\underline{w}$ be as above and let $\mfc$ and
$\mfl$ be defined from $\lambda$ and $\underline{w}$ as in~\eqref{eq:def cij rep} and~\eqref{eq:def ell
  rep}. If the corresponding Grossberg-Karshon twisted cube $(C(\mfc,
\mfl), \rho)$ is untwisted, then:
\begin{enumerate}
\item If $\alpha_i$ appears in the
 word decomposition
  $\underline{w}$, then $\lambda_i \geq 0$. In particular, if all the
  simple roots $\{\alpha_1, \ldots, \alpha_r\}$ appear at least once
  in $\underline{w}$, then $\lambda$ is dominant.
\item If $\alpha_i$ appears more than once in $\underline{w}$, then
  $\lambda_i = 0$.
\end{enumerate}

\end{theorem}

\begin{proof}

Since we assume that $(C(\mfc, \mfl), \rho)$ is untwisted, from
Theorem~\ref{theorem:positivity} we know that $\ell_j \geq 0$ for all $1 \leq j \leq
n$. By the definition of the $\ell_j$ in~\eqref{eq:def ell rep}, the
first statement immediately follows.

Now suppose $\alpha_i$ appears more than once in $\underline{w}$,
i.e. there exist distinct $j < k$ with $\beta_j = \beta_k =
\alpha_i$. Consider the point $(x_{j+1}, \ldots, x_k, \ldots, x_n) =
(0, \ldots, \ell_k, \ldots, 0)$ where all coordinates are $0$ except
the $x_k$ coordinate, which is equal to $\ell_k = \lambda_i$. By
assumption, $(C(\mfc, \mfl), \rho)$ is untwisted, so condition (P)
holds. We claim that the chosen point $(x_{j+1}, \ldots, x_k,
\ldots,x_n)$ satisfies the hypotheses in the statement of condition
(P-j). To see this, first observe that the last $n-k$ coordinates are
all $0$ and $A_k(0,\ldots, 0) = \ell_k$, so the last $n-k+1$
coordinates satisfy the hypotheses for condition (P-(k-1)). The
conclusion of (P-(k-1)) then states that $A_{k-1}(\ell_k, 0, \ldots,0)
\geq 0$. Hence $(x_{k-1}=0, x_k = \ell_k, 0, \ldots, 0)$ satisfies the
hypotheses for the next condition (P-(k-2)). Proceeding similarly
shows that $(x_{j+1}, \ldots, x_k, \ldots, x_n)$ satisfies the
hypotheses of (P-j). Then the conclusion of (P-j) states that $A_j(0,
\ldots,\ell_k, 0 \ldots, 0) \geq 0$. By definition
\begin{equation}
  \label{eq:1}
  \begin{split}
    A_j(0, \ldots, \ell_k, \ldots, 0) & = \ell_j - c_{jk} x_k \\
 & = \ell_j - \langle \beta_k, \beta_j \rangle \ell_k \\
 & = \lambda_i - \langle \alpha_i, \alpha_i^{\vee} \rangle \lambda_i
 \\
 & = \lambda_i - 2 \lambda_i \\
 & = - \lambda_i.
  \end{split}
\end{equation}
Hence $\lambda_i \leq 0$. But we already knew $\lambda_i \geq 0$, so
we conclude $\lambda_i =0$, as required.
\end{proof}

The converse of Theorem~\ref{theorem:necessary conditions} does not hold, as can be seen
by the following. We thank Eunjeong Lee for
providing this example.

\begin{example}\label{example:eunjeong} Consider $G=\mr{SL}_4(\C)$.
Choose
$\underline{w}=s_{\al_2}s_{\al_1}s_{\al_2}s_{\al_3}s_{\al_2}s_{\al_1}$
and $\la=2\varpi_3$. Then the corresponding choices of constants are
\[\mfc=\begin{array}{cccccc}
c_{12} & c_{13} & c_{14} & c_{15} & c_{16}\\
       & c_{23} & c_{24} & c_{25} & c_{26}\\
       &        & c_{34} & c_{35} & c_{36}\\
       &        &        & c_{45} & c_{46}\\
       &        &        &        & c_{56}

  \end{array}=\begin{array}{rrrrrr}
-1 & 2 & -1 & 2 & -1\\
       & -1 & 0 & -1 & 2\\
       &        & -1 & 2 & -1\\
       &        &        & -1 & 0\\
       &        &        &        & -1

  \end{array}\]
and
\[\mfl=\begin{array}{cccccc} \ell_1 & \ell_2 & \ell_3 & \ell_4 &
  \ell_5 & \ell_6 \end{array}=\begin{array}{cccccc} 0& 0 &0 & 2 & 0 &
  0 \end{array}.
\]
The linear functions $A_j$ arising in the definition of the associated
twisted cube are given by
\[\begin{array}{lcrrrrrr}
A_1 & = &0 &+x_2&-2x_3&+x_4&-2x_5&+x_6\\
A_2 & = &0 &&+x_3&&+x_5&-2x_6\\
A_3& = &0&&&+x_4&-2x_5&+x_6\\
A_4 & = & 2&&&&+x_5&\\A_5 & = & 0&&&&&x_6\\
   A_6 & = & 0&&&&\\
\end{array}\]
It can now be seen that for the maximal cone $\si=(-,-,-,-,-,-)$, the
corresponding Cartier data $m_{\si}=(-2,2,2,2,0,0)$ is not contained
in $P_{D(\mfc,\mfl)}$. Thus by Theorem~\ref{theorem:Untwistedness and Base-point free divisor}, the Grossberg-Karshon
twisted cube is \emph{not} untwisted in this case, despite the fact
that both of the conditions stated in Theorem~\ref{theorem:necessary conditions} are satisfied.
\end{example}

\end{document}